\documentclass[a4paper,10pt]{amsart}
\usepackage{amsmath, amssymb, amsthm}
\usepackage{url}
\usepackage{hyperref}
\usepackage{cleveref}
\usepackage{braket}

\theoremstyle{plain}

\newtheorem{thm}{Theorem}[section]
\newtheorem{theorem}[thm]{Theorem}
\newtheorem{lemma}[thm]{Lemma}

\newtheorem{cor}[thm]{Corollary}

\theoremstyle{remark}

\newtheorem{remark}[thm]{Remark}

\newtheorem{example}[thm]{Example}

\newtheorem*{ackn}{Acknowledgment}

\theoremstyle{definition}
\newtheorem{definition}[thm]{Definition}

\crefname{section}{Section}{Sections}

\crefname{thm}{Theorem}{Theorems}
\Crefname{thm}{Theorem}{Theorems}
\crefname{theorem}{Theorem}{Theorems}
\Crefname{theorem}{Theorem}{Theorems}
\crefname{lemma}{Lemma}{Lemmas}
\Crefname{lemma}{Lemma}{Lemmas}
\crefname{corollary}{Corollary}{Corollaries}
\Crefname{corollary}{Corollary}{Corollaries}
\crefname{cor}{Corollary}{Corollaries}
\Crefname{cor}{Corollary}{Corollaries}
\crefname{proposition}{Proposition}{Propositions}
\Crefname{proposition}{Proposition}{Propositions}
\crefname{prop}{Proposition}{Propositions}
\Crefname{prop}{Proposition}{Propositions}
\crefname{remark}{Remark}{Remarks}
\Crefname{remark}{Remark}{Remarks}
\crefname{rem}{Remark}{Remarks}
\Crefname{rem}{Remark}{Remarks}
\crefname{example}{Example}{Examples}
\Crefname{Example}{Example}{Examples}
\crefname{exercise}{Exercise}{Exercises}
\Crefname{exercise}{Exercise}{Exercises}
\crefname{definition}{Definition}{Definitions}
\Crefname{definition}{Definition}{Definitions}
\crefname{dfn}{Definition}{Definitions}
\Crefname{dfn}{Definition}{Definitions}
\crefname{algorithm}{Algorithm}{Algorithms}
\Crefname{algorithm}{Algorithm}{Algorithms}
\crefname{question}{Question}{Questions}
\Crefname{question}{Question}{Questions}
\crefname{problem}{Problem}{Problems}
\Crefname{problem}{Problem}{Problems}
\crefname{notation}{Notation}{Notations}
\Crefname{notation}{Notation}{Notations}
\crefname{conjecture}{Conjecture}{Conjectures}
\Crefname{conjecture}{Conjecture}{Conjectures}
\crefname{conj}{Conjecture}{Conjectures}
\Crefname{conj}{Conjecture}{Conjectures}
\crefname{condition}{Condition}{Conditions}
\Crefname{condition}{Condition}{Conditions}

\newcommand{\ZZ}{\mathbb{Z}}

\newcommand{\NN}{\mathbb{N}}
\newcommand{\RR}{\mathbb{R}}

\newcommand{\KK}{\mathbb{K}}
\newcommand{\AAA}{\mathcal{A}}
\newcommand{\BBB}{\mathcal{B}}

\newcommand{\Der}{\operatorname{Der}}
\newcommand{\Ker}{\operatorname{Ker}}

\makeatletter
\DeclareSymbolFont{symbolsC}{U}{txsyc}{m}{n}
\DeclareMathSymbol{\MYPerp}{\mathrel}{symbolsC}{121}
\makeatother

\newcommand{\der}{\partial}
\newcommand{\numof}[1]{\left|#1\right|}
\newcommand{\defit}[1]{\emph{#1}}

\newcommand{\ldoubleparenthesis}{\mathopen{{(}\!{(}}}
\newcommand{\rdoubleparenthesis}[1][]{\mathclose{{)}\!{)}_{#1}}}
\newcommand{\dps}[2]{\ldoubleparenthesis #1 \rdoubleparenthesis[#2]}

\allowdisplaybreaks[3] 

\begin{document}

\title[A basis for derivations of a Coxeter multiarrangement]
{Explicit description of a basis for derivations of a Coxeter multiarrangement of type $B_2$}

\author[S. Maehara]{Shota Maehara}
\address[S. Maehara]{Joint Graduate School of Mathematics for Innovation, Kyushu University, Fukuoka, Japan.}
\thanks{The first author was supported by WISE program (MEXT) at Kyushu University.}
\email{maehara.shota.027@s.kyushu-u.ac.jp}

\author[Y. Numata]{Yasuhide NUMATA}
\address[Y. Numata]{Department of Mathematics, Hokkaido University, Sapporo, Japan.}
\thanks{The second author was partially supported by JSPS KAKENHI Grant Number JP23K17298.}
\email{nu@math.sci.hokudai.ac.jp}


\begin{abstract}
In this article, we consider the multiarrangements whose underlying arrangements are the Coxeter arrangement of type $B_2$.
For some special multiplicities, we give an explicit description of bases for the derivation modules. 
As an application,
we also describe the lower derivations of bases for
the derivation modules of some Coxeter multiarrangements of type $A_2$, 
which are different from ones given by Wakamiko.
\end{abstract}

\maketitle

\section{introduction}
 \label{sec:introduction}
Let $\KK$ be a field of characteristic zero,
$V$ an $l$-dimensional vector field over $\KK$,
and $S$ the symmetric algebra $S(V^\ast)$ of the dual space $V^{\ast}$ of $V$.
Fixing a basis for $V$,
we identify $V$ with $\KK^l$, and 
$S$ with $\KK[x_1,\ldots,x_l]$.
The free $S$-module
whose basis is $\Set{\der_{x_1},\ldots,\der_{x_l}}$
is denoted by $\Der(S)$,
where
$\der_{x_i}$ stands for the partial derivative operator $\frac{\der}{\der x_i}$
for $x_i$.
Let $\AAA=\Set{H_1,\ldots,H_n}$ be
a \defit{central hyperplane arrangement},
i.e., 
a finite collection of hyperplanes in the vector space $V$
such that $H_i$ contains the origin for all $i$.
Since $H_i$ is a subspace of codimension $1$,
we have $\alpha_i\in S_1$ such that $H_i = \Ker(\alpha_i)$,
where $S_i$ is the space of homogeneous polynomial of degree $i$,
and $\Ker(\alpha)=\Set{v\in V|\alpha(v)=0}$.
For each $i$, fix $\alpha_i$ such that $H_i=\Ker(\alpha_i)$.
For $m=(m_1,\ldots,m_n)\in \NN^n$
we call $m_i$ a \defit{multiplicity} of $H_i$
and $m$ a \defit{multiplicity} on $\AAA$.
We also call a pair $(\AAA,m)$ a \defit{multiarrangement}.
For a multiplicity $m$ on $\AAA$,
the \defit{derivation module} $D(\AAA,m)$ with respect to the multiarrangement $(\AAA,m)$
 is defined to be
the submodule consisting of $\theta\in \Der(S)$
such that $\theta(\alpha_i)$ is in the ideal $\alpha_i^{m_i}S$
of $S$ generated by $\alpha_i^{m_i}$.
The derivation module $D(\AAA,m)$,
introduced in Ziegler \cite{MR1000610},
is one of important algebraic objects
in theory of hyperplane arrangements.

The derivation module $D(\AAA,m)$
is a free $S$-module for some multiarrangement $(\AAA,m)$.
We say that the multiarrangement $(\AAA,m)$ is \defit{free}
if $D(\AAA,m)$ is free.
If $(\AAA,m)$ is free, then
the rank of $D(\AAA,m)$ is the dimension $l$ of $V$,
and
we have
a basis for $D(\AAA,m)$ consisting of homogeneous elements.
For a free multiarrangement $(\AAA,m)$,
the multiset of degrees of elements in a homogeneous basis
does not depend on the choice of a homogeneous basis.
The degrees of elements in a homogeneous basis are called \defit{exponents}
of the multiarrangement $(\AAA,m)$,
and the multiset is denoted by $\exp(\AAA,m)$,
which is one of important statistics for a free multiarrangement $(\AAA,m)$.
In \cite{MR1000610},
Ziegler showed
a criterion to determine the freeness of $D(\AAA,m)$,
which implies that
the sum of exponents is the sum $|m|=\sum_{i=1}^n m_i$ of multiplicities.
The criterion is
a multiarrangement version of Saito's criterion
(see also 
\cite{MR1217488,MR0586450}).
The criterion implies the following:
Let $\theta_i\in D(\AAA,m)$
be a homogeneous element of degree $e_i$.
Assume that $\Set{\theta_1,\ldots,\theta_l}$ is linearly independent over $S$.
The elements $\theta_1,\ldots,\theta_l$ form a basis for $D(\AAA,m)$
if and only if 
$\sum_{i=1}^l e_i=|m|$. 

Consider the case where the dimension $l$ of $V$ is two.
In \cite{MR1000610},
Ziegler showed
that $(\AAA,m)$ are free for all multiplicity $m$.
Hence
to describe exponents and a basis for $D(\AAA,m)$
is one of important topics to study multiarrangements in the two-dimensional vector space $V$.
The general description of
the exponents  is, however, still difficult
as in Wakefield and Yuzvinsky \cite{MR2309190}.
For some specific multiarrangements,
bases for $D(\AAA,m)$ are known.
For example,
a basis for $D(\AAA,m)$ with $\numof{\AAA}=3$
is described with generalized binomial coefficients in \cite{MR2328057}.
It is also described with integral expressions
in \cite{2309.01287}.

The main purpose of this article is
to give an explicit description of a basis for
$D(\AAA,m)$ in the case
where $\AAA$ is the Coxeter arrangement of type $B_2$.
Moreover,
for $D(\AAA,m)$
of a special multiplicity on $\AAA$ with $\numof{\AAA}=3$,
we will give yet another description
of the basis element whose degree is the lowest,
which is unique up to scalar.

This article is organized as follows: 
In \cref{sec:notation}, 
we recall some definitions and theorems about 
multiarrangements on a 2-dimensional vector space,
especially for the Coxeter multiarrangements of type $B_2$.
In \cref{sec:mainresults}, we state the main theorem and some corollaries.
In \cref{sec:proofofmainresults},
for a multiplicity $m=(m_1,m_2,m_3,m_3)$, we give a proof of main results by using double induction 
on $m_1$ and $m_2$.
In \cref{sec:application}, 
as an application of main results,
we obtain yet another expression of the lower derivations of bases for derivation modules for the Coxeter multiarrangements of type $A_2$
for a special case.


\begin{ackn}
The authors are grateful to Takuro Abe 
for many comments on lemmas in this article. 
\end{ackn} 



\section{Notation}
\label{sec:notation}
In this section, we give some definitions and theorems about 
multiarrangements on a 2-dimensional vector space,
especially for the Coxeter multiarrangements of type $B_2$.

Before introducing the Coxeter arrangement of type $B_2$, 
let us fix a polynomial ring $S$ and the module of derivations $\Der(S)$.
Let $\KK$ be a field of characteristic zero, and $S=\KK [x,y]$.
We define $\Der(S)$ by $\Der(S)=\Set{f\der_{x}+g\der_{y}|f,g \in S}$,
where $\der_{x}$ and $\der _{y}$ are partial derivative operators for $x$ and $y$, respectively.
For $\theta = f\der_{x}+g\der_{y} \in \Der(S)$,
we say that $\theta$ is a homogeneous element of degree $d$ 
if $f$ and $g$ are homogeneous polynomials of degree $d$.

As the following definition, we define the Coxeter arrangement of type $B_2$.

\begin{definition}
We define $\alpha _1 = x$, $\alpha _2 = y$, $\alpha _3 = x-y$, $\alpha _4 = x+y$, and
$\AAA = \Set{ \Ker(\alpha_i)\subset \KK^2 | i=1,2,3,4 }$.
We call $\AAA$ \defit{the Coxeter arrangement of type $B_2$}.
\end{definition}
In this article, we consider the Coxeter arrangement $\AAA$ of type $B_2$ with defining polynomials 
$\alpha_1$,$\alpha_2$,$\alpha_3$,$\alpha_4$.

\begin{definition}
 \label{def:D(A)}
For a multiplicity $m=(m_1,m_2,m_3,m_4)$, we define
$D(m)$ to be $D(\AAA,m) 
= \Set{\theta \in \Der (S) |\forall i, \theta(\alpha _{i})\in \alpha _{i}^{m_{i}}S}$. 
Moreover, 
let $(\theta _1, \theta _2)$ be a homogeneous basis for $D(m)$ 
such that the degree of $\theta _i$ is $e_i$. Then, 
we define $\exp(m)$ to be $\exp(\AAA,m)=(e_1,e_2)$.
\end{definition}

We recall Saito's criterion in this case.
\begin{theorem}[Saito's criterion]
 \label{thm:Saito}
For a multiplicity $m=(m_1,m_2,m_3,m_4)$, 
let $\theta_1,\theta_2$ be homogeneous derivations in $D(m)$ 
such that the degree of $\theta_i$ is $e_i$. Then,
$(\theta_1,\theta_2)$ is a basis for $D(m)$ if and only if
they are independent over $S$ and $|m|=m_1+m_2+m_3+m_4=e_1+e_2$. 
\end{theorem}

We use the following lemma to check the independency of derivations.

\begin{lemma}
 \label{lem:divide}
Let $m$ be a multiplicity with
$\exp(m)=(e_{1},e_{2})$, $e_1\leq e_2$ 
and let $\theta_{1},\theta_{2}$ be homogeneous derivations in $D(m)$ such that 
the degree of $\theta_{i}$ is $e_{i}$. Then,
$(\theta_{1},\theta_{2})$ is a basis for $D(m)$ 
if and only if $\theta_{2}$ is not equal to $f\theta_{1}$ for any $f\in S_{e_2-e_1}$.
\end{lemma}
\begin{proof}
Sufficiency is clear.
Assume that $\theta_{2}$ is not equal to $f\theta_{1}$ for any $f\in S_{e_2-e_1}$. 
Since $\exp(m)=(e_1,e_2)$, 
there exists a homogeneous derivation $\theta_{2}' \in D(m)$ such that 
$D(m)=\langle \theta_{1},\theta_{2}' \rangle_{S}$ and the degree of $\theta_{2}'$ is $e_{2}$.
Since
$\theta_{2} \in D(m)$ and the degree of $\theta_2$ is $e_2$, we have
$\theta_{2}=g\theta_{1}+s \theta_{2}'$
with $g \in S_{e_{2}-e_{1}}$ and $s \in \mathbb{K}$.
Moreover, since $\theta_{2}$ is not equal to $g\theta_{1}$ by assumption, 
we have $s$ is not equal to zero.
Hence it follows that $\theta_{2}'=-\frac{g}{s}\theta_{1}+\frac{1}{s}\theta_{2}$, which implies that 
$D(m)=\langle \theta_{1},\theta_{2} \rangle_{S}$.
\end{proof}

A multiplicity is said to be balanced if each element of 
the multiplicity is less than the summation of the other elements.

\begin{definition}
 \label{def:multiplicity}
For a multiplicity $m=(m_1,m_2,m_3,m_4)$, 
we say that $m$ is \defit{balanced} if
$2m_i \leq |m|-1$ for all $i$.
\end{definition}

For a multiplicity $m$ which is not balanced, 
we know the exponents $\exp(m)$ as follows.

\begin{thm}
Let $m=(m_1,m_2,m_3,m_4)$ satisfy $m_k=\max \set{m_1,m_2,m_3,m_4}$.
If $m$ is not balanced, then
$\exp(m)=(|m|-m_k,m_k)$.
\end{thm}

\begin{proof}
Without loss of generality, we can assume that
$m_2=\max \set{m_1,m_2,m_3,m_4}$.

Let $\psi=f\der_{x}+g\der_{y}\in D(m)$ be a  homogeneous  element of degree $n$.
Assume that  $n<m_1+m_3+m_4 \leq m_2$.
Since $y^{m_2}$ divides $\psi(y)=g$ and $n<m_2$, we have $g=0$ and $\psi=f\der_x$.
Hence $\psi(\alpha_i)=f$ for $i=1,3,4$.
Since $\alpha_i^{m_i}$ divides $\psi(\alpha_i)=f$ for $i=1,3,4$, 
it follows that 
$\alpha_1^{m_1}\alpha_3^{m_3}\alpha_4^{m_4}$ divides $f$. 
Since $n<m_1+m_3+m_4$, we have $f=0$ and $\psi=0$.
Hence $D(m)$ has no nonzero element whose degree is less than $m_1+m_3+m_4=|m|-m_2$.

Let $\theta=\alpha_1^{m_1}\alpha_3^{m_3}\alpha_4^{m_4}\der_{x}$.
It is clear that  
$\theta$ is in $D(m)$, and 
that the degree is $m_1+m_3+m_4=|m|-m_2$.
Hence $\theta$ has the lowest degree in $D(m)$ and
$\exp(m)=(|m|-m_2,m_2)$ by \Cref{thm:Saito}.
\end{proof}

\section{Main Results}
\label{sec:mainresults}
In this section, we state the main theorem and some corollaries.

Before statements, let us introduce the double Pochhammer symbol, which can be considered as 
an extension of the double factorial.
For $a\in \RR$ and $n\in \NN$,
we define the \defit{double Pochhammer symbol} $\dps{a}{n}$
by
\begin{align*}
  \dps{a}{n}=
  \begin{cases}
    1 & (n=0)\\
    (a+2(n-1))\cdot \dps{a}{n-1} &(n>0).
  \end{cases}
\end{align*}
For $n\in\ZZ_{>0}$,
\begin{align*}
  \dps{a}{n}&=\prod_{i=0}^{n-1} (a+2i)=a(a+2)\cdots (a+2(n-1)),\\
  \dps{1}{n}&=(2n-1)!!,\\
  \dps{2}{n}&=(2n)!!,
\end{align*}
where $m!!$ stands for the double factorial of $m$.

\begin{definition}
 \label{def:main}
Let $m=(m_1,m_2,m_3,m_3)$ satisfy
$m_1$,$m_2\in 2\ZZ+1$, $|m|\in 4\ZZ$, and let $d=\frac{|m|}{2}-1$.
We define $f_m,g_m,\theta_m$ by
\begin{align*}
f_{m}&=
\sum_{i=0}^{\frac{d-m_{1}}{2}} (-1)^i 
\frac{
\dps{m_{1}+m_{2}-d}
{\frac{d-m_{1}}{2}-i}
}
{(d-m_{1}-2i)!!(d-2i)!!(2i)!!}
x^{d-2i}y^{2i}\\
&=
x^{m_1}\cdot
\sum_{i=0}^{\frac{d-m_{1}}{2}} (-1)^i 
\frac{
\dps{m_{1}+m_{2}-d}
{\frac{d-m_{1}}{2}-i}
}
{(d-m_{1}-2i)!!(d-2i)!!(2i)!!}
x^{d-m_1-2i}y^{2i},\\
g_{m}&=
\sum_{i=0}^{\frac{d-m_{2}}{2}} (-1)^{i+m_{3}} 
\frac{
\dps{m_{1}+m_{2}-d}
{\frac{d-m_{2}}{2}-i}
}
{(d-m_{2}-2i)!!(d-2i)!!(2i)!!}
x^{2i}y^{d-2i}\\
&=
y^{m_2}\cdot
\sum_{i=0}^{\frac{d-m_{2}}{2}} (-1)^{i+m_{3}} 
\frac{
\dps{m_{1}+m_{2}-d}
{\frac{d-m_{2}}{2}-i}
}
{(d-m_{2}-2i)!!(d-2i)!!(2i)!!}
x^{2i}y^{d-m_2-2i},\\
\theta_{m}&=f_{m}\der_{x} -g_{m}\der_{y}.
\end{align*}
If $\frac{d-m_1}{2}<0$ and $\frac{d-m_2}{2}<0$, 
then we define $f_m$ and $g_m$ to be zero, respectively.
\end{definition}

\begin{remark}
Since $|m|\in 4\ZZ \subset 2\ZZ$, 
$\frac{d-m_i}{2}=\frac{1}{2}((\frac{|m|}{2}-1)-m_i) \geq 0$ 
if and only if  $2m_i \leq |m|-1$.
Therefore by \Cref{def:multiplicity}, 
if $m$ is balanced, then $\frac{d-m_i}{2}\geq0$.
\end{remark}

The following is the main theorem of this article.
\begin{theorem}[Main theorem]
  \label{thm:main}
Let $m=(m_1,m_2,m_3,m_3)$ satisfy
$m_1$,$m_2\in 2\ZZ+1$, $|m|\in 4\ZZ$.
If $m$ satisfies $2m_i \leq |m|+2$ for all $i$, then
$\theta_{m}\in D(m)$,
and $\exp(m)=(\frac{|m|}{2}-1,\frac{|m|}{2}+1)$.
Moreover, if $m$ is balanced, then
$D(m)=\langle \theta_{m}, \theta_{m''} \rangle_{S}$,
where $m''=(m_1+2,m_2+2,m_3,m_3)$.
\end{theorem}

We have some corollaries to \Cref{thm:main}. 
\begin{cor}
 \label{cor:Abe}
Let $m=(m_1,m_2,m_3,m_4)$ be a balanced multiplicity that satisfies $|m_{1}-m_{2}| \geq |m_{3}-m_{4}|$.
Regarding the difference of exponents, we have the following: 
\begin{enumerate}
\item
If $|m|$ is even and the following condition are satisfied,
then
the difference of exponents is two:
\begin{enumerate}
\item \label{cond:1}
  \label{cond:1a}
  $m_{1},m_{2}\in2\ZZ+1$,
  $|m|\in 4\ZZ$,
  $m_{3}=m_{4}$.
\end{enumerate}
\item \label{cond:2}
If $|m|$ is even and one of the following are satisfied,
then the difference of exponents is zero:
\begin{enumerate}
\item
  Condition \ref{cond:1a} are not satisfied and
  $m_{3}=m_{4}$.
\item \label{cond:2b}
  $m_{3}=m_{4}-1$. 
\item
  $m_{1},m_{2}\in2\ZZ+1$ and  $m_{3}=m_{4}-2$.
\end{enumerate}
\item \label{cond:3}
If $|m|$ is odd and one of the following are satisfied,
then the difference of exponents is one:
\begin{enumerate}
\item \label{cond:3a}
  $m_{3}=m_{4}$.
\item
  $m_{3}=m_{4}-1$. 
\item
  $m_{3}=m_{4}-2$. 
\item
  $m_{1},m_{2}\in2\ZZ+1$ and  $m_{3}=m_{4}-3$.
\end{enumerate}
\end{enumerate}
\end{cor}

\begin{proof}
Case \ref{cond:1a} is the direct consequence of \Cref{thm:main}.
Case \ref{cond:2} and Case \ref{cond:3} are derived from Case \ref{cond:1a} by \cite{MR2873095}.
\end{proof}

\begin{remark}
Abe conjectured \Cref{cor:Abe}, and proved 
the case where $m_3$ is large enough 
in \cite{MR2510979}.
\end{remark}

\begin{cor}
  \label{cor:P}
Let $m=(m_1,m_2,m_3,m_3)$ be a balanced multiplicity that satisfies 
$m_1$,$m_2\in 2\ZZ+1$ and $|m|\in 4\ZZ$. 
Then, $\alpha_i$ does not divide $\theta_{m}$ for $i=1,2,3,4$.
\end{cor}
\begin{proof}
Assume that $\alpha_1$ divides $\theta_{m}$.
Let $m^{(-1)}=(m_1-1,m_2,m_3,m_3)$.
By \Cref{cor:Abe}-\ref{cond:3a}, $\exp(m^{(-1)})=(\frac{|m|}{2}-1,\frac{|m|}{2})$.
On the other hand, since $\theta_{m}$ is a derivation of degree $\frac{|m|}{2}-1$ 
in $D(m)$, 
$\frac{1}{\alpha_1}\theta_{m}$ is a derivation of degree $\frac{|m|}{2}-2$
in $D(m^{(-1)})$.
This is contradiction. We can also prove that 
$\alpha_i$ does not divide $\theta_{m}$ for $i=2,3,4$ similarly.
\end{proof}

We can calculate an explicit basis for $D(m)$ for some special $m$.

\begin{cor}
 \label{cor:Q}
For a multiplicity $\mu=(\mu_{1},\mu_{2},\mu_{3},\mu_{3})$,
define $\theta(\mu)$ by $\theta(\mu)=\theta_\mu$, where 
$\theta_\mu$ is defined in \Cref{def:main}.
For a balanced multiplicity $m$, 
we can calculate a basis for $D(m)$ as follows:

\begin{enumerate}
\item 
The case where $m_3=m_4$

 \begin{enumerate}

\item 
If $|m|\in 4\ZZ$, we have the following:

  \begin{enumerate}
\item 
\label{cond:1ai}
If $m_1,m_2 \in 2\ZZ+1$, then \\
$D(m)=\langle \theta(m), \theta(m+(2,2,0,0)) \rangle_{S}$.

\item 
If $m_1,m_2 \in 2\ZZ$, then \\
$D(m)=\langle x\theta(m+(-1,1,0,0)), y\theta(m+(1,-1,0,0)) \rangle_{S}$.
  \end{enumerate}

\item 
If $|m|\in 4\mathbb{Z}+1$, we have the following:

  \begin{enumerate}
\item 
If $m_{1}\in2\ZZ$ and $m_{2}\in2\ZZ+1$, then \\ 
$D(m)=\langle x\theta(m+(-1,0,0,0)), \theta(m+(1,0,1,1)) \rangle_{S}$.
  \end{enumerate}

\item 
If $|m|\in 4\ZZ+2$, we have the following:
  \begin{enumerate}
\item 
\label{cond:1ci}
If $m_1,m_2 \in 2\ZZ+1$, then \\
$D(m)=\langle \theta(m+(2,0,0,0)), \theta(m+(0,2,0,0)) \rangle_{S}$.

\item 
If $m_1,m_2 \in 2\ZZ$, then \\
$D(m)=\langle \theta(m+(1,1,0,0)),xy\theta(m+(-1,-1,0,0)) \rangle_{S}$.
  \end{enumerate}

\item 
If $|m|\in 4\mathbb{Z}+3$, we have the following:

  \begin{enumerate}
\item 
\label{cond:1di}
If $m_{1}\in2\ZZ$ and $m_{2}\in2\ZZ+1$, then \\ 
$D(m)=\langle \theta(m+(1,0,0,0)), x\theta(m+(-1,0,1,1)) \rangle_{S}$.
  \end{enumerate}

 \end{enumerate}

\item 
The case where $m_3=m_4-1$

 \begin{enumerate} 
\item 
If $|m|\in 4\ZZ$, we have the following:

  \begin{enumerate}
\item 
If $m_1\in2\ZZ$ and $m_2 \in 2\ZZ+1$, then \\
$D(m)=\langle x\theta(m+(-1,0,1,0)), (x+y)\theta(m+(1,0,0,-1)) \rangle_{S}$.
  \end{enumerate}

\item 
If $|m|\in 4\ZZ+1$, we have the following:

  \begin{enumerate}
\item 
If $m_1,m_2 \in 2\ZZ+1$, then \\
$D(m)=\langle (x+y)\theta(m+(0,0,0,-1)), \theta(m+(2,0,1,0)) \rangle_{S}$.
  \end{enumerate}

\item 
If $|m|\in 4\ZZ+2$, we have the following:

  \begin{enumerate}
\item 
If $m_1\in2\ZZ$ and $m_2 \in 2\ZZ+1$, then \\
$D(m)=\langle \theta(m+(1,0,1,0)), x(x+y)\theta(m+(-1,0,0,-1)) \rangle_{S}$.
  \end{enumerate}

\item 
If $|m|\in 4\ZZ+3$, we have the following:

  \begin{enumerate}
\item 
If $m_1,m_2 \in 2\ZZ+1$, then \\
$D(m)=\langle \theta(m+(0,0,1,0)), (x-y)(x+y)^2\theta(m+(0,0,-1,-2)) \rangle_{S}$.
  \end{enumerate}

 \end{enumerate}

\item 
The case where $m_3=m_4-2$ 
 \begin{enumerate} 
\item 
If $|m|\in 4\ZZ+2$, we have the following:

  \begin{enumerate}
\item 
If $m_1,m_2 \in 2\ZZ+1$, then \\
$D(m)=\langle \theta(m+(0,0,2,0)), (x+y)^2\theta(m+(0,0,0,-2)) \rangle_{S}$.
  \end{enumerate}

 \end{enumerate}

\end{enumerate}
\end{cor}

\begin{proof}
Case \ref{cond:1ai} is the direct consequence of \Cref{thm:main}.
For the other cases, by \Cref{lem:divide}, 
it is sufficient to show that each pair of derivations are not equal even up to scalar.
For Case \ref{cond:1ci},
assume that 
$\theta(m+(2,0,0,0))$ is equal to $\theta(m+(0,2,0,0))$ up to scalar.
Then, $\theta(m+(2,0,0,0))$ is a derivation of degree $\frac{|m|}{2}$ 
in $D(m+(2,2,0,0))$.
On the other hand, $\exp(m+(2,2,0,0))=(\frac{|m|}{2}+2,\frac{|m|}{2}+2)$ by \Cref{cor:Abe}.
This is contradiction.
The proof of Case \ref{cond:1di} is almost similar.
For the other cases, 
\Cref{cor:P} completes the proof.
\end{proof}

\begin{remark}
In \cite{MR3070120},
Abe has proved the following:
Let $\BBB$ be an arrangement in the $2$-dimensional vector space over a field of characteristic zero 
and 
$\nu$ be a multiplicity on $\BBB$. 
If $\nu$ is balanced,  
then the difference of exponents of $(\BBB,\nu)$ is less than or equal to 
$|\BBB|-2$.
Since we assume that 
$|\AAA|-2=4-2=2$, 
we obtain a special case of this theorem from \Cref{cor:Abe}.
\end{remark}

We give some examples of $\theta_m$.
\begin{example}
For some $m$, the derivations $\theta_m$ are the following:
\begin{align*}
    \theta_{(3,5,2,2)}&=
    \textstyle\left(\frac{1}{10} \, x^{5} - \frac{1}{6} \, x^{3} y^{2}\right)\partial_x+\left(-\frac{1}{15} \, y^{5}\right)\partial_y
,
\\
    \theta_{(3,5,4,4)}&=
    \textstyle\left(\frac{1}{280} \, x^{7} - \frac{1}{60} \, x^{5} y^{2} + \frac{1}{24} \, x^{3} y^{4}\right)\partial_x+\left(\frac{1}{30} \, x^{2} y^{5} - \frac{1}{210} \, y^{7}\right)\partial_y
,
\\
    \theta_{(3,9,4,4)}&=
    \textstyle\left(\frac{1}{432} \, x^{9} - \frac{1}{112} \, x^{7} y^{2} + \frac{1}{80} \, x^{5} y^{4} - \frac{1}{144} \, x^{3} y^{6}\right)\partial_x+\left(-\frac{1}{945} \, y^{9}\right)\partial_y
,
\\
    \theta_{(5,7,2,2)}&=
    \textstyle\left(\frac{1}{42} \, x^{7} - \frac{1}{30} \, x^{5} y^{2}\right)\partial_x+\left(-\frac{1}{105} \, y^{7}\right)\partial_y
,
\\
    \theta_{(5,7,4,4)}&=
    \textstyle\left(\frac{1}{504} \, x^{9} - \frac{1}{140} \, x^{7} y^{2} + \frac{1}{120} \, x^{5} y^{4}\right)\partial_x+\left(\frac{1}{210} \, x^{2} y^{7} - \frac{1}{630} \, y^{9}\right)\partial_y
,
\\
    \theta_{(5,11,4,4)}&=
    \textstyle\left(\frac{1}{1584} \, x^{11} - \frac{1}{432} \, x^{9} y^{2} + \frac{1}{336} \, x^{7} y^{4} - \frac{1}{720} \, x^{5} y^{6}\right)\partial_x+\left(-\frac{1}{10395} \, y^{11}\right)\partial_y
,
\\
    \theta_{(7,9,2,2)}&=
    \textstyle\left(\frac{1}{270} \, x^{9} - \frac{1}{210} \, x^{7} y^{2}\right)\partial_x+\left(-\frac{1}{945} \, y^{9}\right)\partial_y
,
\\
    \theta_{(7,9,4,4)}&=
    \textstyle\left(\frac{1}{2376} \, x^{11} - \frac{1}{756} \, x^{9} y^{2} + \frac{1}{840} \, x^{7} y^{4}\right)\partial_x+\left(\frac{1}{1890} \, x^{2} y^{9} - \frac{1}{4158} \, y^{11}\right)\partial_y
,
\\
    \theta_{(7,13,4,4)}&=
    \textstyle\left(\frac{1}{9360} \, x^{13} - \frac{1}{2640} \, x^{11} y^{2} + \frac{1}{2160} \, x^{9} y^{4} - \frac{1}{5040} \, x^{7} y^{6}\right)\partial_x+\left(-\frac{1}{135135} \, y^{13}\right)\partial_y

.
\end{align*}
\end{example}

\begin{example}
For some $m$ with $m_1=1$, the derivations $\theta_m$ are the following:
\begin{align*}
    \theta_{(1,3,2,2)}&=
    \textstyle\left(\frac{1}{6} \, x^{3} - \frac{1}{2} \, x y^{2}\right)\partial_x+\left(-\frac{1}{3} \, y^{3}\right)\partial_y
,
\\
    \theta_{(1,5,3,3)}&=
    \textstyle\left(\frac{1}{40} \, x^{5} - \frac{1}{12} \, x^{3} y^{2} + \frac{1}{8} \, x y^{4}\right)\partial_x+\left(\frac{1}{15} \, y^{5}\right)\partial_y
,
\\
    \theta_{(1,3,4,4)}&=
    \textstyle\left(-\frac{1}{120} \, x^{5} + \frac{1}{12} \, x^{3} y^{2} + \frac{1}{8} \, x y^{4}\right)\partial_x+\left(\frac{1}{6} \, x^{2} y^{3} + \frac{1}{30} \, y^{5}\right)\partial_y
,
\\
    \theta_{(1,7,4,4)}&=
    \textstyle\left(\frac{1}{336} \, x^{7} - \frac{1}{80} \, x^{5} y^{2} + \frac{1}{48} \, x^{3} y^{4} - \frac{1}{48} \, x y^{6}\right)\partial_x+\left(-\frac{1}{105} \, y^{7}\right)\partial_y
,
\\
    \theta_{(1,5,5,5)}&=
    \textstyle\left(-\frac{1}{1680} \, x^{7} + \frac{1}{240} \, x^{5} y^{2} - \frac{1}{48} \, x^{3} y^{4} - \frac{1}{48} \, x y^{6}\right)\partial_x+\left(-\frac{1}{30} \, x^{2} y^{5} - \frac{1}{210} \, y^{7}\right)\partial_y
,
\\
    \theta_{(1,9,5,5)}&=
    \textstyle\left(\frac{1}{3456} \, x^{9} - \frac{1}{672} \, x^{7} y^{2} + \frac{1}{320} \, x^{5} y^{4} - \frac{1}{288} \, x^{3} y^{6} + \frac{1}{384} \, x y^{8}\right)\partial_x+\left(\frac{1}{945} \, y^{9}\right)\partial_y
,
\\
    \theta_{(1,3,6,6)}&=
    \textstyle\left(\frac{1}{1680} \, x^{7} - \frac{1}{80} \, x^{5} y^{2} - \frac{1}{16} \, x^{3} y^{4} - \frac{1}{48} \, x y^{6}\right)\partial_x\\
    &\textstyle+\left(-\frac{1}{24} \, x^{4} y^{3} - \frac{1}{20} \, x^{2} y^{5} - \frac{1}{280} \, y^{7}\right)\partial_y
,
\\
    \theta_{(1,7,6,6)}&=
    \textstyle\left(-\frac{1}{24192} \, x^{9} + \frac{1}{3360} \, x^{7} y^{2} - \frac{1}{960} \, x^{5} y^{4} + \frac{1}{288} \, x^{3} y^{6} + \frac{1}{384} \, x y^{8}\right)\partial_x\\
    &\textstyle+\left(\frac{1}{210} \, x^{2} y^{7} + \frac{1}{1890} \, y^{9}\right)\partial_y
,
\\
    \theta_{(1,11,6,6)}&=
    \textstyle\left(\frac{1}{42240} \, x^{11} - \frac{1}{6912} \, x^{9} y^{2} + \frac{1}{2688} \, x^{7} y^{4} - \frac{1}{1920} \, x^{5} y^{6} + \frac{1}{2304} \, x^{3} y^{8}\right.\\
    &\textstyle\left.-\frac{1}{3840} \, x y^{10}\right)\partial_x+\left(-\frac{1}{10395} \, y^{11}\right)\partial_y
,
\\
    \theta_{(1,5,7,7)}&=
    \textstyle\left(\frac{1}{40320} \, x^{9} - \frac{1}{3360} \, x^{7} y^{2} + \frac{1}{320} \, x^{5} y^{4} + \frac{1}{96} \, x^{3} y^{6} + \frac{1}{384} \, x y^{8}\right)\partial_x\\
    &\textstyle+\left(\frac{1}{120} \, x^{4} y^{5} + \frac{1}{140} \, x^{2} y^{7} + \frac{1}{2520} \, y^{9}\right)\partial_y

.
\end{align*}
\end{example}

\section{Proof of Main Results}
\label{sec:proofofmainresults}
In this section, for a multiplicity $m=(m_1,m_2,m_3,m_3)$, we give a proof of main results by using double induction on $m_1$ and $m_2$.

The following lemma is one of base cases for the induction on $m_1$.
We consider the case where a multiplicity is not balanced but close to a balanced multiplicity.

\begin{lemma}
  \label{lem:F}
Let $m=(m_1,m_2,m_3,m_3)$ satisfy $m_1$,$m_2 \in 2\ZZ +1$ and $|m|\in4\ZZ$.
If $m$ is not balanced but satisfies $2m_i \leq |m|+2$ for all $i$, then
$\theta_{m} \in D(m)$.
\end{lemma}
\begin{proof}
Without loss of generality, we can assume that
$m_2=m_1+2m_3+2$.
It is obvious that $|m|=2m_1+4m_3+2$ and $d=m_1+2m_3$. By \Cref{def:main},
\begin{align*}
f_{m}
&=
x^{m_1}\cdot
\sum_{i=0}^{m_3} (-1)^i
\frac{
\dps{m_1+2}
{m_3-i}
}
{(2m_3-2i)!!(m_1+2m_3-2i)!!(2i)!!}
x^{2m_3-2i}y^{2i}.
\end{align*}
Since
$(2k)!!
=(2k)(2k-2)\cdots4\cdot2
=2^k\cdot k!$
for all $k \in \ZZ_{>0}$,
\begin{align*}
\frac{
\dps{m_1+2}
{m_3-i}
}
{(2m_3-2i)!!(m_1+2m_3-2i)!!(2i)!!}
&=
\frac{
(m_1+2)(m_1+4)\cdots(m_1+2m_3-2i)
}
{2^{m_3-i}(m_3-i)!(m_1+2m_3-2i)!!2^i\cdot i!}\\
&=
\frac{1}{2^{m_3} \cdot m_1!!(m_3-i)!i!}.
\end{align*}
Hence
\begin{align*}
f_{m}
&=
x^{m_1}\cdot
\sum_{i=0}^{m_3} (-1)^i
\frac{1}{2^{m_3} \cdot m_1!!(m_3-i)!i!}
x^{2m_3-2i}y^{2i}\\
&=
\frac{x^{m_1}}
{m_1!!(2^{m_3} \cdot m_3!)}
\cdot
\sum_{i=0}^{m_3} (-1)^i
\frac{m_3!}{(m_3-i)!i!}
x^{2m_3-2i}y^{2i}\\
&=
\frac{x^{m_1}}
{m_1!!(2m_3)!!}
\cdot
\sum_{i=0}^{m_3}
\binom{m_3}{i}
(x^2)^{m_3-i}(-y^2)^i\\
&=
\frac{x^{m_1}}
{m_1!!(2m_3)!!}
\cdot
(x^2-y^2)^{m_3}\\
&=
\frac{1}
{m_1!!(2m_3)!!}
\cdot
x^{m_1}
(x-y)^{m_3}(x+y)^{m_3}.
\end{align*}
On the other hand, since
$d-m_2=(m_1+2m_3)-(m_1+2m_3+2)
=-2$,
\begin{align*}
g_{m}
=
\sum_{i=0}^{\frac{d-m_{2}}{2}} (-1)^{i+m_{3}} 
\frac{
\dps{m_{1}+m_{2}-d}
{\frac{d-m_{2}}{2}-i}
}
{(d-m_{2}-2i)!!(d-2i)!!(2i)!!}
x^{2i}y^{d-2i}
=0.
\end{align*}
Therefore,
\begin{align*}
\theta_{m}
&=
f_{m}\der_{x} -g_{m}\der_{y}\\
&=
\frac{1}
{m_1!!(2m_3)!!}
\cdot
x^{m_1}
(x-y)^{m_3}(x+y)^{m_3}
\der _{x}\\
&=
\frac{1}
{m_1!!(2m_3)!!}
\cdot
\alpha_1^{m_1}
\alpha_3^{m_3}\alpha_4^{m_3}
\der _{x}
\in D(m).
\end{align*}
\end{proof}

The following lemma is the base case for the induction on $m_2$.

\begin{lemma}
  \label{lem:A}
For $m=(1,1,d,d)$ with $d \in 2\ZZ +1$, 
$\theta_{m} \in D(m)$.
\end{lemma}
\begin{proof}
If $d =1$, we have 
$\theta_m = x\partial_x+y\partial_y \in D(m)$.
Hence we only consider $d \geq 3$.
For $m=(1,1,d,d)$,
\begin{align*}
f_{m}
&=
\sum_{i=0}^{\frac{d-1}{2}} (-1)^i 
\frac{
\dps{2-d}
{\frac{d-1}{2}-i}
}
{(d-1-2i)!!(d-2i)!!(2i)!!}
x^{d-2i}y^{2i}.
\end{align*}
Since $(d-1-2i)!!(d-2i)!!=(d-2i)!$,
\begin{align*}
f_{m}
&=
\sum_{i=0}^{\frac{d-1}{2}} (-1)^i 
\frac{
\dps{2-d}
{\frac{d-1}{2}-i}
}
{(d-2i)!(2i)!!}
x^{d-2i}y^{2i}.
\end{align*}
Since $0\leq i \leq \frac{d-1}{2}$,
$2-d+2((\frac{d-1}{2}-i)-1)=-2i-1 \leq-1$.
Since $2-d$ is also negative, we have 
\begin{align*}
\dps{2-d}{\frac{d-1}{2}-i}
&=(-1)^{\frac{d-1}{2}-i}\dps{2i+1}{\frac{d-1}{2}-i}\\
&=(-1)^{\frac{d-1}{2}-i}\frac{(d-2)!!}{(2i-1)!!}\\
&=(-1)^{\frac{d-1}{2}-i}\frac{(d-1)!}{(2i-1)!!(d-1)!!}\\
&=(-1)^{\frac{d-1}{2}-i}\frac{d!}{(2i-1)!!(d-1)!!d}.
\end{align*}
Hence 
\begin{align*}
f_{m}
&=
\sum_{i=0}^{\frac{d-1}{2}} (-1)^i 
\frac{1}
{(d-2i)!(2i)!!}\cdot
(-1)^{\frac{d-1}{2}-i}\frac{d!}{(2i-1)!!(d-1)!!d}
x^{d-2i}y^{2i}\\
&=
(-1)^{\frac{d-1}{2}}
\frac{1}{(d-1)!!d}\cdot
\sum_{i=0}^{\frac{d-1}{2}}
\frac{d!}
{(d-2i)!(2i)!}
x^{d-2i}y^{2i}\\
&=
(-1)^{\frac{d-1}{2}}
\frac{1}
{(d-1)!!d}\cdot
\sum_{i=0}^{\frac{d-1}{2}} 
\binom{d}{2i}
x^{d-2i}y^{2i}.
\end{align*}
On the other hand, 
\begin{align*}
g_{m}(x,y)&=
\sum_{i=0}^{\frac{d-1}{2}} (-1)^{i+d} 
\frac{
\dps{2-d}
{\frac{d-1}{2}-i}
}
{(d-1-2i)!!(d-2i)!!(2i)!!}
x^{2i}y^{d-2i}\\
&=
-\sum_{i=0}^{\frac{d-1}{2}} (-1)^{i} 
\frac{
\dps{2-d}
{\frac{d-1}{2}-i}
}
{(d-1-2i)!!(d-2i)!!(2i)!!}
x^{2i}y^{d-2i}\\
&=
-f_{m}(y,x)\\
&=
-(-1)^{\frac{d-1}{2}}
\frac{1}
{(d-1)!!d}\cdot
\sum_{i=0}^{\frac{d-1}{2}}
\binom{d}{2i}
x^{2i}y^{d-2i}.
\end{align*}
Therefore, it follows that 
\begin{align*}
\theta_{m}(x-y)=
f_{m}+g_{m}&=(-1)^{\frac{d-1}{2}}
\frac{1}
{(d-1)!!d}\cdot(x-y)^{d},\\
\theta_{m}(x+y)=
f_{m}-g_{m}&=(-1)^{\frac{d-1}{2}}
\frac{1}
{(d-1)!!d}\cdot(x+y)^{d}.
\end{align*}
Since $x^{m_1}$ divides $\theta_m(x)=f_m$ and $y^{m_2}$divides $\theta_m(y)=-g_m$ by \Cref{def:main}, 
we have $\theta_m \in D(m)$. 
\end{proof}

\Cref{lem:B,lem:C} are essential for the induction step on $m_2$.

\begin{lemma}
  \label{lem:B}
Let $m=(1,m_2,m_3,m_3)$ satisfy $m_2 \in 2\ZZ+1$ and $|m|\in 4\ZZ$. Moreover, 
let $m'=(1,m_2,m_3+2,m_3+2)$ and $m''=(1,m_2+2,m_3+1,m_3+1)$.
If $m$ is balanced, then
\begin{align*}
e\theta_{m''}
&=
b_0\theta_{m'}
-d_0(x^2-y^2)\theta_{m}\\
&=
b_0\theta_{m'}
-d_0\alpha_3\alpha_4\theta_{m},
\end{align*}
where
\begin{align*}
b_0 &= (-1)^{\frac{d-m_2}{2}+m_3}\frac{1}{m_2 !!(d-m_2)!!},\\
d_0 &= (-1)^{\frac{d-m_2}{2}+m_3+1}\frac{1}{m_2 !!(d+2-m_2)!!},\\
e   &= (-1)^{\frac{d-m_2}{2}+m_3+1}\frac{m_2-2d-3}{m_2!!(d+2-m_2)!!} \ \ \text{for}\\ 
d   &= \frac{|m|}{2}-1.
\end{align*}
\end{lemma}
\begin{proof}
Note that $m'$ and $m''$ are balanced for a balanced multiplicity $m$, 
and that we have 
$b_0 = (m_2-d-2)d_0$ and $e = (m_2-2d-3)d_0$ by definition.

By \Cref{def:main},
\begin{align*}
f_{m}&=
\sum_{i=0}^{\frac{d-1}{2}} (-1)^i 
\frac{
\dps{1+m_{2}-d}
{\frac{d-1}{2}-i}
}
{(d-2i)!(2i)!!}
x^{d-2i}y^{2i},\\
f_{m'}&=
\sum_{i=0}^{\frac{d+1}{2}} (-1)^i 
\frac{
\dps{-1+m_{2}-d}
{\frac{d+1}{2}-i}
}
{(d+2-2i)!(2i)!!}
x^{d+2-2i}y^{2i},\\
f_{m''}&=
\sum_{i=0}^{\frac{d+1}{2}} (-1)^i 
\frac{
\dps{1+m_{2}-d}
{\frac{d+1}{2}-i}
}
{(d+2-2i)!(2i)!!}
x^{d+2-2i}y^{2i}.
\end{align*}
It follows that
\begin{align*}
&(x^2-y^2)f_{m}\\
&=
\sum_{i=0}^{\frac{d-1}{2}} (-1)^i 
\frac{
\dps{1+m_{2}-d}
{\frac{d-1}{2}-i}
}
{(d-2i)!(2i)!!}
x^{d-2i}y^{2i}
(x^2-y^2)\\
&=
\sum_{i=0}^{\frac{d-1}{2}} (-1)^i 
\frac{
\dps{1+m_{2}-d}
{\frac{d-1}{2}-i}
}
{(d-2i)!(2i)!!}
x^{d+2-2i}y^{2i}\\
&-
\sum_{i=0}^{\frac{d-1}{2}} (-1)^i 
\frac{
\dps{1+m_{2}-d}
{\frac{d-1}{2}-i}
}
{(d-2i)!(2i)!!}
x^{d-2i}y^{2i+2}\\
&=
\sum_{i=0}^{\frac{d-1}{2}} (-1)^i 
\frac{
\dps{1+m_{2}-d}
{\frac{d-1}{2}-i}
}
{(d-2i)!(2i)!!}
x^{d+2-2i}y^{2i}\\
&-
\sum_{j=1}^{\frac{d+1}{2}} (-1)^{j-1} 
\frac{
\dps{1+m_{2}-d}
{\frac{d-1}{2}-j+1}
}
{(d-2j+2)!(2j-2)!!}
x^{d+2-2j}y^{2j}\\
&=
\sum_{i=1}^{\frac{d-1}{2}} (-1)^i 
\left(
\frac{
\dps{1+m_{2}-d}
{\frac{d-1}{2}-i}
}
{(d-2i)!(2i)!!}
+
\frac{
\dps{1+m_{2}-d}
{\frac{d-1}{2}-i+1}
}
{(d-2i+2)!(2i-2)!!}
\right)
x^{d+2-2i}y^{2i}\\
&+
\frac{
\dps{1+m_2-d}
{\frac{d-1}{2}}
}
{d!}
x^{d+2}
+
(-1)^{\frac{d+1}{2}}
\frac{1}
{(d-1)!!}
xy^{d+1}\\
&=
\sum_{i=1}^{\frac{d-1}{2}}
(-1)^i 
\frac{
\dps{1+m_{2}-d}
{\frac{d-1}{2}-i}
}
{(d+2-2i)!(2i)!!}
(
(d+2-2i)(d+1-2i)
+
2i(m_2-2i)
)
x^{d+2-2i}y^{2i}\\
&+
\frac{
\dps{1+m_2-d}
{\frac{d-1}{2}}
}
{d!}
x^{d+2}
+
(-1)^{\frac{d+1}{2}}
\frac{1}
{(d-1)!!}
xy^{d+1}\\
&=
\sum_{i=0}^{\frac{d+1}{2}}
(-1)^i
\frac{
\dps{1+m_{2}-d}
{\frac{d-1}{2}-i}
}
{(d+2-2i)!(2i)!!}
(
(d+2)(d+1)
+2i(m_2-2d-3)
)
x^{d+2-2i}y^{2i}.
\end{align*}
Since
\begin{align*}
&(d+2-m_2)(-1+m_2-d)
+
(d+2)(d+1)
+2i(m_2-2d-3)\\
&=
(d+2)m_2+m_2(d+1)-m_2^2
+2i(m_2-2d-3)\\
&=
-m_2^2+(2d+3)m_2+2i(m_2-2d-3)\\
&=
(m_2-2i)(-m_2+2d+3),
\end{align*}
we have
\begin{align*}
&(m_2-d-2)f_{m'}
-(x^2-y^2)f_{m}\\
&=
(m_2-d-2)\cdot
\sum_{i=0}^{\frac{d+1}{2}} (-1)^i 
\frac{
\dps{-1+m_{2}-d}
{\frac{d+1}{2}-i}
}
{(d+2-2i)!(2i)!!}
x^{d+2-2i}y^{2i}\\
&-
\sum_{i=0}^{\frac{d+1}{2}}
(-1)^i
\frac{
\dps{1+m_{2}-d}
{\frac{d-1}{2}-i}
}
{(d+2-2i)!(2i)!!}
(
(d+2)(d+1)+2i(m_2-2d-3)
)
x^{d+2-2i}y^{2i}\\
&=
-
\sum_{i=0}^{\frac{d+1}{2}}
(-1)^i
\frac{
\dps{1+m_2-d}{\frac{d-1}{2}-i}
}
{(d+2-2i)!(2i)!!}
(
(d+2-m_2)(-1+m_2-d)
+
(d+2)(d+1)\\
&+
2i(m_2-2d-3)
)
x^{d+2-2i}y^{2i}\\
&=
-
\sum_{i=0}^{\frac{d+1}{2}}
(-1)^i
\frac{
\dps{1+m_2-d}{\frac{d-1}{2}-i}
}
{(d+2-2i)!(2i)!!}
(m_2-2i)(-m_2+2d+3)
x^{d+2-2i}y^{2i}\\
&=
(m_2-2d-3)\cdot
\sum_{i=0}^{\frac{d+1}{2}}(-1)^{i}
\frac{\dps{1+m_{2}-d}
{\frac{d+1}{2}-i}}
{(d+2-2i)!(2i)!!}
x^{d+2-2i}y^{2i}\\
&=
(m_2-2d-3)f_{m''}.
\end{align*}
Therefore, 
\begin{align*}
b_0f_{m'}-d_0(x^2-y^2)f_{m}
=
d_0((m_2-d-2)f_{m'}-(x^2-y^2)f_{m})
=ef_{m''}.
\end{align*}

On the other hand, by \Cref{def:main},
\begin{align*}
g_{m}&=
\sum_{i=0}^{\frac{d-m_{2}}{2}} (-1)^{i+m_{3}} 
\frac{
\dps{1+m_{2}-d}
{\frac{d-m_{2}}{2}-i}
}
{(d-m_{2}-2i)!!(d-2i)!!(2i)!!}
x^{2i}y^{d-2i}\\
g_{m'}&=
\sum_{i=0}^{\frac{d-m_{2}}{2}+1} (-1)^{i+m_{3}+2} 
\frac{
\dps{-1+m_{2}-d}
{\frac{d-m_{2}}{2}-i+1}
}
{(d+2-m_{2}-2i)!!(d+2-2i)!!(2i)!!}
x^{2i}y^{d+2-2i}\\
g_{m''}&=
\sum_{i=0}^{\frac{d-m_{2}}{2}} (-1)^{i+m_{3}+1} 
\frac{
\dps{1+m_{2}-d}
{\frac{d-m_{2}}{2}-i}
}
{(d-m_{2}-2i)!!(d+2-2i)!!(2i)!!}
x^{2i}y^{d+2-2i}.
\end{align*}
It follows that
\begin{align*}
&(x^2-y^2)g_{m}
=
\sum_{i=0}^{\frac{d-m_{2}}{2}} (-1)^{i+m_{3}} 
\frac{
\dps{1+m_{2}-d}
{\frac{d-m_{2}}{2}-i}
}
{(d-m_{2}-2i)!!(d-2i)!!(2i)!!}
x^{2i}y^{d-2i}
(-y^2+x^2)\\
&=
-
\sum_{i=0}^{\frac{d-m_{2}}{2}} (-1)^{i+m_{3}} 
\frac{
\dps{1+m_{2}-d}
{\frac{d-m_{2}}{2}-i}
}
{(d-m_{2}-2i)!!(d-2i)!!(2i)!!}
x^{2i}y^{d+2-2i}\\
&+
\sum_{i=0}^{\frac{d-m_{2}}{2}} (-1)^{i+m_{3}} 
\frac{
\dps{1+m_{2}-d}
{\frac{d-m_{2}}{2}-i}
}
{(d-m_{2}-2i)!!(d-2i)!!(2i)!!}
x^{2i+2}y^{d-2i}\\
&=
-
\sum_{i=0}^{\frac{d-m_{2}}{2}} (-1)^{i+m_{3}} 
\frac{
\dps{1+m_{2}-d}
{\frac{d-m_{2}}{2}-i}
}
{(d-m_{2}-2i)!!(d-2i)!!(2i)!!}
x^{2i}y^{d+2-2i}\\
&+
\sum_{j=1}^{\frac{d-m_{2}}{2}+1} (-1)^{j-1+m_{3}} 
\frac{
\dps{1+m_{2}-d}
{\frac{d-m_{2}}{2}-j+1}
}
{(d-m_{2}-2j+2)!!(d-2j+2)!!(2j-2)!!}
x^{2j}y^{d-2j+2}\\
&=
-
\sum_{i=1}^{\frac{d-m_{2}}{2}} (-1)^{i+m_{3}} 
\biggl(
\frac{
\dps{1+m_{2}-d}
{\frac{d-m_{2}}{2}-i}
}
{(d-m_{2}-2i)!!(d-2i)!!(2i)!!}\\
&+
\frac{
\dps{1+m_{2}-d}
{\frac{d-m_{2}}{2}-i+1}
}
{(d-m_{2}-2i+2)!!(d-2i+2)!!(2i-2)!!}
\biggr)
x^{2i}y^{d-2i+2}\\
&-
(-1)^{m_3}
\frac{
\dps{1+m_2-d}{\frac{d-m_2}{2}}
}
{(d-m_2)!!d!!}
y^{d+2}
+
(-1)^{\frac{d-m_{2}}{2}+m_{3}} 
\frac{1}
{m_2!!(d-m_2)!!}
x^{d-m_2+2}y^{m_2}\\
&=
-
\sum_{i=1}^{\frac{d-m_{2}}{2}}
\frac{
(-1)^{i+m_{3}}
\dps{1+m_{2}-d}
{\frac{d-m_{2}}{2}-i}
}
{(d-m_{2}-2i+2)!!(d-2i+2)!!(2i)!!}
(
(d-m_{2}-2i+2)(d-2i+2)\\
&+
2i(1-2i)
)
x^{2i}y^{d-2i+2}
-
(-1)^{m_3}
\frac{
\dps{1+m_2-d}{\frac{d-m_2}{2}}
}
{(d-m_2)!!d!!}
y^{d+2}\\
&+
(-1)^{\frac{d-m_{2}}{2}+m_{3}} 
\frac{1}
{m_2!!(d-m_2)!!}
x^{d-m_2+2}y^{m_2}\\
&=
-
\sum_{i=0}^{\frac{d-m_{2}+1}{2}}
\frac{
(-1)^{i+m_{3}}
\dps{1+m_{2}-d}
{\frac{d-m_{2}}{2}-i}
}
{(d+2-m_{2}-2i)!!(d+2-2i)!!(2i)!!}
(
(d+2-m_2)(d+2)\\
&-
2i(2d+3-m_{2})
)
x^{2i}y^{d-2i+2}.
\end{align*}
Since
\begin{align*}
&(m_2-d-2)
(-1+m_{2}-d)
+(d+2-m_2)(d+2)
-2i(2d+3-m_{2})\\
&=
(d+2-m_2)(2d+3-m_2)
-2i(2d+3-m_2)\\
&=
(d+2-m_2-2i)(2d+3-m_2),
\end{align*}
we have
\begin{align*}
&(m_2-d-2)g_{m'}
-(x^2-y^2)g_{m}\\
&=
(m_2-d-2)\cdot
\sum_{i=0}^{\frac{d-m_{2}}{2}+1} (-1)^{i+m_{3}} 
\frac{
\dps{-1+m_{2}-d}
{\frac{d-m_{2}}{2}-i+1}
}
{(d+2-m_{2}-2i)!!(d+2-2i)!!(2i)!!}
x^{2i}y^{d+2-2i}\\
&+
\sum_{i=0}^{\frac{d-m_{2}}{2}+1}
(-1)^{i+m_{3}}
\frac{
\dps{1+m_{2}-d}
{\frac{d-m_{2}}{2}-i}
}
{(d+2-m_{2}-2i)!!(d+2-2i)!!(2i)!!}
(
(d+2-m_2)(d+2)\\
&-
2i(2d+3-m_{2})
)
x^{2i}y^{d-2i+2}\\
&=
\sum_{i=0}^{\frac{d-m_2}{2}+1}
\frac{(-1)^{i+m_3}
\dps{1+m_{2}-d}
{\frac{d-m_{2}}{2}-i}
}
{(d+2-m_{2}-2i)!!(d+2-2i)!!(2i)!!}
(
(m_2-d-2)
(-1+m_{2}-d)\\
&+
(d+2-m_2)(d+2)
-
2i(2d+3-m_{2})
)
x^{2i}y^{d+2-2i}\\
&=
\sum_{i=0}^{\frac{d-m_2}{2}}(-1)^{i+m_3}
\frac{(2d+3-m_2)
\dps{1+m_{2}-d}
{\frac{d-m_{2}}{2}-i}
}
{(d-m_2-2i)!!(d+2-2i)!!(2i)!!}
x^{2i}y^{d+2-2i}\\
&=
(m_2-2d-3)\cdot
\sum_{i=0}^{\frac{d-m_2}{2}}(-1)^{i+m_3+1}
\frac{
\dps{1+m_{2}-d}
{\frac{d-m_{2}}{2}-i}
}
{(d-m_2-2i)!!(d+2-2i)!!(2i)!!}
x^{2i}y^{d+2-2i}\\
&=(m_2-2d-3)g_{m''}.
\end{align*}
Therefore, 
\begin{align*}
b_0g_{m'}-d_0(x^2-y^2)g_{m}
=
d_0((m_2-d-2)g_{m'}-(x^2-y^2)g_{m})
=eg_{m''}.
\end{align*}
Hence
\begin{align*}
b_0\theta_{m'}-d_0(x^2-y^2)\theta_{m}
&=b_0(f_{m'}\der_{x}-g_{m'}\der_{y})
-d_0(x^2-y^2)(f_{m}\der_{x}-g_{m}\der_{y})\\
&=(b_0f_{m'}-d_0(x^2-y^2)f_{m})\der_{x}
-(b_0g_{m'}-d_0(x^2-y^2)g_{m})\der_{y}\\
&=
ef_{m''}\der_{x}-eg_{m''}\der_{y}\\
&=e\theta_{m''}.
\end{align*}
\end{proof}

\begin{lemma}
  \label{lem:C}
Let $m=(1,m_2,m_3,m_3)$ satisfy $m_2 \in 2\ZZ+1$ and $|m|\in 4\ZZ$.
Let $m'=(1,m_2,m_3+2,m_3+2)$ and $m''=(1,m_2+2,m_3+1,m_3+1)$.
If $m$ is balanced, $\theta _{m} \in D(m)$ and $\theta _{m'} \in D(m')$, 
then we have $\theta_{m''} \in D(m'')$.
\end{lemma}
\begin{proof}
Since 
$\theta_{m} \in D(m)$, 
$\alpha_3\alpha_4\theta_{m} \in D((1,m_2,m_3+1,m_3+1))$. 
Moreover, since
$\theta_{m'} \in D(m') \subset D((1,m_2,m_3+1,m_3+1))$, 
$b_0\theta_{m'}-d_0\alpha_3\alpha_4\theta_{m} \in D((1,m_2,m_3+1,m_3+1))$.
Hence, by \Cref{lem:B},
$\theta_{m''} \in D((1,m_2,m_3+1,m_3+1))$.
By \Cref{def:main},
$y^{m_2+2}$ 
divides 
$\theta_{m''}(y) = -g_{m''}$.
Therefore
$\theta_{m''} \in D(m'')$.
\end{proof}

The following is a corollary to \Cref{lem:C}.

\begin{cor}
  \label{cor:C'}
Let $\mu=(1,\mu_2-2,\mu_3-1,\mu_3-1)$ satisfy $\mu_2 \in 2\ZZ+1$, $|\mu|\in 4\ZZ$ and 
$\mu_2-2,\mu_3-1\geq1$.
Let $\mu'=(1,\mu_2-2,\mu_3+1,\mu_3+1)$ and $\mu''=(1,\mu_2,\mu_3,\mu_3)$.
If $\mu$ is balanced, $\theta _{\mu} \in D(\mu)$ and $\theta _{\mu'} \in D(\mu')$, 
then we have $\theta_{\mu''} \in D(\mu'')$.
\end{cor}

The following lemma is not only the conclusion of the induction on $m_2$, 
but also one of base cases for the induction on $m_1$.

\begin{lemma}
  \label{lem:G}
Let $m=(1,m_2,m_3,m_3)$ satisfy
$m_2\in 2\ZZ+1$, $|m|\in 4\ZZ$.
If $m$ is balanced, then
$\theta_{m}\in D(m)$.
\end{lemma}
\begin{proof}
We show the lemma by the induction on $m_2$.

First, we consider the base case.
As \Cref{lem:A}, we have already shown the case where $m_2=1$.

Next, we consider the induction step.
Assume that
$\theta_{\mu} \in D(\mu)$ 
for all balanced multiplicity $\mu=(1,\mu_2,\mu_3,\mu_3)$ 
satisfying $\mu_2 \leq 2k-1$, $\mu_2 \in 2\ZZ+1$, $|\mu| \in 4\ZZ$.
Let $m''=(1,2k+1,l,l)$ be a multiplicity such that
$k,l\in \ZZ_{>0}$, $|m''|=2(k+l+1)\in 4\ZZ$ and $m''$ is balanced.
By assumptions, we have $k+l\in 2\ZZ+1$ and $2k+1 \leq 2l$, especially $2 \leq l$.
Then, we can consider two balanced multiplicities
$m=(1,2k-1,l-1,l-1)$ and $m'=(1,2k-1,l+1,l+1)$.
By induction hypothesis, applying \Cref{cor:C'}
to $m$ and $m'$, $\theta_{m''} \in D(m'')$.
\end{proof}

\Cref{lem:D,lem:E} are essential for the induction step on $m_1$.

\begin{lemma}
  \label{lem:D}
Let $m=(m_1,m_2,m_3,m_3)$ satisfy $m_1$,$m_2 \in 2\ZZ+1$ and $|m|\in 4\ZZ$. Moreover, 
let $m'=(m_1,m_2+4,m_3,m_3)$ and $m''=(m_1+2,m_2+2,m_3,m_3)$.
If $m$ is balanced, then
\begin{align*}
e\theta_{m''}&=
b_0\theta_{m'}
-d_0y^2\theta_{m}\\
&=
b_0\theta_{m'}
-d_0\alpha_2^2\theta_{m},
\end{align*}
where
\begin{align*}
b_0 &= (-1)^{\frac{d-m_1}{2}}\frac{1}{m_1 !!(d-m_1)!!},\\
d_0 &= (-1)^{\frac{d-m_1}{2}+1}\frac{1}{m_1 !!(d+2-m_1)!!},\\
e   &= (-1)^{\frac{d-m_1}{2}}\frac{m_2+2}{m_1!!(d+2-m_1)!!} \ \ \text{for}\\ 
d   &= \frac{|m|}{2}-1.
\end{align*}
\end{lemma}
\begin{proof}
Let $m'=(m_1,m_2+4,m_3,m_3)=(m'_1,m'_2,m'_3,m'_3)$.
Note that $m'$ is not balanced but satisfies $2m'_i \leq |m'|+2$ for all $i$
if and only if $m_2=m_1+2m_3-2$,
and that we have 
$b_0 = (m_1-d-2)d_0$ and $e = -(m_2+2)d_0$ by definition.

By \Cref{def:main},
\begin{align*}
f_{m}&=
\sum_{i=0}^{\frac{d-m_{1}}{2}} (-1)^i 
\frac{
\dps{m_{1}+m_{2}-d}
{\frac{d-m_{1}}{2}-i}
}
{(d-m_{1}-2i)!!(d-2i)!!(2i)!!}
x^{d-2i}y^{2i},\\
f_{m'}&=
\sum_{i=0}^{\frac{d+2-m_{1}}{2}} (-1)^i 
\frac{
\dps{m_{1}+m_{2}-d+2}
{\frac{d+2-m_{1}}{2}-i}
}
{(d+2-m_{1}-2i)!!(d+2-2i)!!(2i)!!}
x^{d+2-2i}y^{2i},\\
f_{m''}&=
\sum_{i=0}^{\frac{d-m_{1}}{2}} (-1)^i 
\frac{
\dps{m_{1}+m_{2}-d+2}
{\frac{d-m_{1}}{2}-i}
}
{(d-m_{1}-2i)!!(d+2-2i)!!(2i)!!}
x^{d+2-2i}y^{2i}.
\end{align*}
It follows that
\begin{align*}
y^2f_{m}
&=
\sum_{i=0}^{\frac{d-m_{1}}{2}} (-1)^i 
\frac{
\dps{m_{1}+m_{2}-d}
{\frac{d-m_{1}}{2}-i}
}
{(d-m_{1}-2i)!!(d-2i)!!(2i)!!}
x^{d-2i}y^{2i+2}\\
&=
\sum_{j=1}^{\frac{d-m_{1}}{2}+1} (-1)^{j-1} 
\frac{
\dps{m_{1}+m_{2}-d}
{\frac{d-m_{1}}{2}-j+1}
}
{(d+2-m_{1}-2j)!!(d+2-2j)!!(2j-2)!!}
x^{d+2-2j}y^{2j}.
\end{align*}
Since
\begin{align*}
&(m_1-d-2)(m_2+2-2i)
+
2i(m_1+m_2-d)\\
&=
(m_1-d-2)((m_2+2)-2i)
+
2i((m_2+2)-2+m_1-d)\\
&=
(m_2+2)((m_1-d-2)+2i)
-2i(m_1-d-2)+2i(-2+m_1-d)\\
&=
-(m_2+2)(d+2-m_1-2i),
\end{align*}
we have
\begin{align*}
&(m_1-d-2)f_{m'}
-y^2f_{m}\\
&=
(m_1-d-2)\cdot
\sum_{i=0}^{\frac{d-m_{1}}{2}+1} (-1)^i 
\frac{
\dps{m_{1}+m_{2}-d+2}
{\frac{d-m_{1}}{2}-i+1}
}
{(d+2-m_{1}-2i)!!(d+2-2i)!!(2i)!!}
x^{d+2-2i}y^{2i}\\
&+
\sum_{j=1}^{\frac{d-m_{1}}{2}+1} (-1)^{j} 
\frac{
\dps{m_{1}+m_{2}-d}
{\frac{d-m_{1}}{2}-j+1}
}
{(d+2-m_{1}-2j)!!(d+2-2j)!!(2j-2)!!}
x^{d+2-2j}y^{2j}\\
&=
\sum_{i=1}^{\frac{d-m_{1}}{2}+1} 
\frac{(-1)^i
\dps{m_{1}+m_{2}-d+2}
{\frac{d-m_{1}}{2}-i}
}
{(d+2-m_{1}-2i)!!(d+2-2i)!!(2i)!!}
(
(m_1-d-2)(m_2+2-2i)\\
&+
2i(m_1+m_2-d)
)
x^{d+2-2i}y^{2i}
+
(m_1-d-2)
\frac{
\dps{m_1+m_2-d+2}{\frac{d-m_1}{2}+1}
}
{(d+2-m_1)!!(d+2)!!}
x^{d+2}\\
&=
-\sum_{i=1}^{\frac{d-m_{1}}{2}} 
\frac{(-1)^i
\dps{m_{1}+m_{2}-d+2}
{\frac{d-m_{1}}{2}-i}
(m_2+2)(d+2-m_1-2i)}
{(d+2-m_{1}-2i)!!(d+2-2i)!!(2i)!!}
x^{d+2-2i}y^{2i}\\
&-
(d+2-m_1)
\frac{
\dps{m_1+m_2-d+2}{\frac{d-m_1}{2}}
((m_1+m_2-d+2)+2\cdot \frac{d-m_1}{2})
}
{(d+2-m_1)!!(d+2)!!}
x^{d+2}\\
&=
-(m_2+2)\cdot
\sum_{i=0}^{\frac{d-m_{1}}{2}} 
(-1)^i
\frac{
\dps{m_{1}+m_{2}-d+2}
{\frac{d-m_{1}}{2}-i}
}
{(d-m_{1}-2i)!!(d+2-2i)!!(2i)!!}
x^{d+2-2i}y^{2i}.
\end{align*}
Therefore
\begin{align*}
&b_0f_{m'}
-d_0y^2f_{m}\\
&=
d_0
((m_1-d-2)f_{m'}
-y^2f_{m})\\
&=
-d_0(m_2+2)\cdot
\sum_{i=0}^{\frac{d-m_{1}}{2}}
(-1)^i
\frac{\dps{m_{1}+m_{2}-d+2}
{\frac{d-m_{1}}{2}-i}}
{(d-m_1-2i)!!(d+2-2i)!!(2i)!!}
x^{d-2i+2}y^{2i}\\
&=
ef_{m''}.
\end{align*}

On the other hand, by \Cref{def:main}
\begin{align*}
g_{m}&=
\sum_{i=0}^{\frac{d-m_{2}}{2}} (-1)^{i+m_{3}} 
\frac{
\dps{m_{1}+m_{2}-d}
{\frac{d-m_{2}}{2}-i}
}
{(d-m_{2}-2i)!!(d-2i)!!(2i)!!}
x^{2i}y^{d-2i},\\
g_{m'}&=
\sum_{i=0}^{\frac{d-2-m_{2}}{2}} (-1)^{i+m_{3}} 
\frac{
\dps{m_{1}+m_{2}-d+2}
{\frac{d-2-m_{2}}{2}-i}
}
{(d-2-m_{2}-2i)!!(d+2-2i)!!(2i)!!}
x^{2i}y^{d+2-2i},\\
g_{m''}&=
\sum_{i=0}^{\frac{d-m_{2}}{2}} (-1)^{i+m_{3}} 
\frac{
\dps{m_{1}+m_{2}-d+2}
{\frac{d-m_{2}}{2}-i}
}
{(d-m_{2}-2i)!!(d+2-2i)!!(2i)!!}
x^{2i}y^{d+2-2i}.
\end{align*}
If $m'$ is balanced, 
since
\begin{align*}
&(m_1-d-2)(d-m_2-2i)-(d+2-2i)(m_1+m_2-d)\\
&=
(m_1-d-2)((d-2i+2)-(m_2+2))-(d-2i+2)((m_2+2)+(m_1-d-2))\\
&=
(m_2+2)
((d+2-m_1)-(d+2-2i))\\
&=
-(m_2+2)(m_1-2i),
\end{align*}
we have
\begin{align*}
&(m_1-d-2)g_{m'}
-y^2g_{m}\\
&=
(m_1-d-2)\cdot
\sum_{i=0}^{\frac{d-2-m_{2}}{2}} (-1)^{i+m_{3}} 
\frac{
\dps{m_{1}+m_{2}-d+2}
{\frac{d-2-m_{2}}{2}-i}
}
{(d-2-m_{2}-2i)!!(d+2-2i)!!(2i)!!}
x^{2i}y^{d+2-2i}\\
&-
\sum_{i=0}^{\frac{d-m_{2}}{2}} (-1)^{i+m_{3}} 
\frac{
\dps{m_{1}+m_{2}-d}
{\frac{d-m_{2}}{2}-i}
}
{(d-m_{2}-2i)!!(d-2i)!!(2i)!!}
x^{2i}y^{d+2-2i}\\
&=
\sum_{i=0}^{\frac{d-m_{2}}{2}-1} (-1)^{i+m_{3}} 
\frac{
\dps{m_{1}+m_{2}-d+2}
{\frac{d-m_{2}}{2}-i-1}
}
{(d-m_{2}-2i)!!(d+2-2i)!!(2i)!!}
(
(m_1-d-2)(d-m_2-2i)\\
&-
(d+2-2i)(m_1+m_2-d)
)
x^{2i}y^{d+2-2i}
-
\frac{
(-1)^{\frac{d-m_2}{2}+m_3}
}
{m_2!!(d-m_2)!!}
x^{d-m_2}y^{m_2+2}\\
&=
-
(m_2+2)\cdot
\sum_{i=0}^{\frac{d-m_{2}}{2}-1} (-1)^{i+m_{3}} 
\frac{
\dps{m_{1}+m_{2}-d+2}
{\frac{d-m_{2}}{2}-i}
}
{(d-m_{2}-2i)!!(d+2-2i)!!(2i)!!}
x^{2i}y^{d+2-2i}\\
&-
(m_2+2)
(-1)^{\frac{d-m_2}{2}+m_3}
\frac{1}
{(m_2+2)!!(d-m_2)!!}
x^{d-m_2}y^{m_2+2}\\
&=
-
(m_2+2)\cdot
\sum_{i=0}^{\frac{d-m_{2}}{2}} (-1)^{i+m_{3}} 
\frac{
\dps{m_{1}+m_{2}-d+2}
{\frac{d-m_{2}}{2}-i}
}
{(d-m_{2}-2i)!!(d+2-2i)!!(2i)!!}
x^{2i}y^{d+2-2i}.
\end{align*}
Therefore
\begin{align*}
&b_0g_{m'}
-d_0y^2g_{m}\\
&=
d_0(
(m_1-d-2)g_{m'}
-y^2g_{m})\\
&=
-d_0
(m_2+2)\cdot
\sum_{i=0}^{\frac{d-m_{2}}{2}} (-1)^{i+m_{3}} 
\frac{
\dps{m_{1}+m_{2}-d+2}
{\frac{d-m_{2}}{2}-i}
}
{(d-m_{2}-2i)!!(d+2-2i)!!(2i)!!}
x^{2i}y^{d+2-2i}\\
&=
eg_{m''}.
\end{align*}
If $m'$ is not balanced, 
since $d=m_2=m_1+2m_3-2$, 
\begin{align*}
&b_0g_{m'}
-d_0y^2g_{m}
=-d_0y^2g_{m}\\
&=
-d_0\cdot
\sum_{i=0}^{\frac{d-m_{2}}{2}} (-1)^{i+m_{3}} 
\frac{
\dps{m_{1}+m_{2}-d}
{\frac{d-m_{2}}{2}-i}
}
{(d-m_{2}-2i)!!(d-2i)!!(2i)!!}
x^{2i}y^{d+2-2i}\\
&=
-d_0
(-1)^{m_3}
\frac{1}
{d!!}
y^{d+2}\\
&=
-d_0(d+2)
(-1)^{m_3}
\frac{1}
{(d+2)!!}
y^{d+2}\\
&=eg_{m''}.
\end{align*}

Therefore, whether $m'$ is balanced or not,
\begin{align*}
b_0\theta_{m'}-d_0y^2\theta_{m}
&=b_0(f_{m'}\der_{x}-g_{m'}\der_{y})
-d_0y^2(f_{m}\der_{x}-g_{m}\der_{y})\\
&=(b_0f_{m'}-d_0y^2f_{m})\der_{x}
-(b_0g_{m'}-d_0y^2g_{m})\der_{y}\\
&=
ef_{m''}\der_{x}-eg_{m''}\der_{y}\\
&=e\theta_{m''}.
\end{align*}
\end{proof}

\begin{lemma}
  \label{lem:E}
Let $m=(m_1,m_2,m_3,m_3)$ satisfy $m_1$,$m_2 \in 2\ZZ+1$ and $|m|\in 4\ZZ$.
Let $m'=(m_1,m_2+4,m_3,m_3)$ and $m''=(m_1+2,m_2+2,m_3,m_3)$.
If $m$ is balanced, $\theta _{m} \in D(m)$ and $\theta _{m'} \in D(m')$, 
then we have $\theta_{m''} \in D(m'')$.
\end{lemma}

\begin{proof}
Since
$\theta_{m} \in D(m)$, $\alpha_2^2\theta_{m} \in D((m_1,m_2+2,m_3,m_3))$.
Moreover, since
$\theta_{m'} \in D(m') \subset D((m_1,m_2+2,m_3,m_3))$, 
$b_0\theta_{m'}-d_0\alpha_2^2\theta_{m} \in D((m_1,m_2+2,m_3,m_3))$.
Hence, by \Cref{lem:D},
$\theta_{m''} \in D((m_1,m_2+2,m_3,m_3))$.
By \Cref{def:main},
$x^{m_1+2}$
divides
$\theta_{m''}(x) = f_{m''}.$ 
Therefore
$\theta_{m''} \in D(m'')$.
\end{proof}

The following is a corollary to \Cref{lem:E}.

\begin{cor}
  \label{cor:E'}
Let $\mu=(\mu_1-2,\mu_2-2,\mu_3,\mu_3)$ satisfy $\mu_1$,$\mu_2 \in 2\ZZ+1$, $|\mu|\in 4\ZZ$ 
and $\mu_1-2,\mu_2-2\geq1$.
Let $\mu'=(\mu_1-2,\mu_2+2,\mu_3,\mu_3)$ and $\mu''=(\mu_1,\mu_2,\mu_3,\mu_3)$.
If $\mu$ is balanced, $\theta _{\mu} \in D(\mu)$ and $\theta _{\mu'} \in D(\mu')$, 
then we have $\theta_{\mu''} \in D(\mu'')$.
\end{cor}

At the end of this section, we finish the proof of \Cref{thm:main} 
by showing \Cref{lem:I,lem:M}.

\begin{lemma}
  \label{lem:I}
Let $m=(m_1,m_2,m_3,m_3)$ satisfy
$m_1$,$m_2\in 2\ZZ+1$, $|m|\in 4\ZZ$.
If $m$ satisfies $2m_i \leq |m|+2$ for all $i$, then
$\theta_{m}\in D(m)$.
\end{lemma}

\begin{proof}
We show the lemma by the induction on $m_1$.

First, we consider the base case.
As \Cref{lem:G,lem:F}, we have already shown the case where $m_1=1$ 
whether $m$ is balanced or not.

Next, we consider the induction step.
Assume that
$\theta_{\mu} \in D(\mu)$ 
for all multiplicity $\mu=(\mu_1,\mu_2,\mu_3,\mu_3)$ 
that satisfy 
$2m_i \leq |m|+2$ for all $i$, $\mu_1 \leq 2k-1$, $\mu_1,\mu_2 \in 2\ZZ+1$ and $|\mu| \in 4\ZZ$.
Let $m''=(2k+1,2n+1,l,l)$ be a multiplicity such that
$k,n,l\in \ZZ_{>0}$, $|m''|=2(k+n+l+1)\in 4\ZZ$ and $m''$ is balanced.
By assumptions, we have $k+n+l\in 2\ZZ+1$ and $2n+1 \leq 2k+2l$. Hence, since
$2n+3 \leq 2(k+l+1)$ 
if and only if
$2n+3 \leq 2(k+l+1)-1
=(2k-1)+l+l+2$,
we can consider two multiplicities
$m=(2k-1,2n-1,l,l)$ and $m'=(2k-1,2n+3,l,l)=(m'_1,m'_2,m'_3,m'_3)$, where 
$m$ is balanced and $m'$ satisfy $2m'_i \leq |m'|+2$ for all $i$.
By induction hypothesis, applying \Cref{cor:E'}
to $m$ and $m'$, $\theta_{m''} \in D(m'')$.
Note that by \Cref{lem:F}, 
if $m''$ is not balanced but satisfies $2m''_i \leq |m''|+2$ for all $i$, 
$\theta_{m''}\in D(m'')$.

Therefore we complete the proof.
\end{proof}

\begin{lemma}
  \label{lem:M}
Let $m=(m_1,m_2,m_3,m_3)$ satisfy
$m_1$,$m_2\in 2\ZZ+1$, $|m|\in 4\ZZ$.
Let $m''=(m_1+2,m_2+2,m_3,m_3)$.
If $m$ is balanced, then
$D(m)
=\langle \theta_{m}, \theta_{m''} \rangle_{S}$.
\end{lemma}

\begin{proof}
First, it is obvious that
$\theta_{m},\theta_{m''} \in D(m)$
by \Cref{lem:I}.

Second, to prove the independence of $\theta_{m}$ and $\theta_{m''}$,
consider the following matrix $M$:

\[
M=
\left(
\begin{array}{cc}
f_{m} & f_{m''}\\
g_{m} & g_{m''}
\end{array}
\right),\\
\]
where

\begin{align*}
f_{m}&=
\sum_{i=0}^{\frac{d-m_{1}}{2}} (-1)^i 
\frac{
\dps{m_{1}+m_{2}-d}
{\frac{d-m_{1}}{2}-i}
}
{(d-m_{1}-2i)!!(d-2i)!!(2i)!!}
x^{d-2i}y^{2i},\\
f_{m''}&=
\sum_{i=0}^{\frac{d-m_{1}}{2}} (-1)^i 
\frac{
\dps{m_{1}+m_{2}-d+2}
{\frac{d-m_{1}}{2}-i}
}
{(d-m_{1}-2i)!!(d+2-2i)!!(2i)!!}
x^{d+2-2i}y^{2i},\\
g_{m}&=
\sum_{i=0}^{\frac{d-m_{2}}{2}} (-1)^{i+m_{3}} 
\frac{
\dps{m_{1}+m_{2}-d}
{\frac{d-m_{2}}{2}-i}
}
{(d-m_{2}-2i)!!(d-2i)!!(2i)!!}
x^{2i}y^{d-2i},\\
g_{m''}&=
\sum_{i=0}^{\frac{d-m_{2}}{2}} (-1)^{i+m_{3}} 
\frac{
\dps{m_{1}+m_{2}-d+2}
{\frac{d-m_{2}}{2}-i}
}
{(d-m_{2}-2i)!!(d+2-2i)!!(2i)!!}
x^{2i}y^{d+2-2i}.
\end{align*}
Since
the degree of $f_{m}g_{m''}$ with respect to $x$ is $2d-m_2$ and
the degree of $f_{m''}g_{m}$ with respect to $x$ is $2d-m_2+2$, 
we have 
$\det M = f_{m}g_{m''}-f_{m''}g_{m} \neq 0$.
Therefore, 
$\theta_{m}$ and $\theta_{m''}$ are independent over $S$.

Hence, by \Cref{thm:Saito}, we have
$D(m)
=\langle \theta_{m}, \theta_{m''} \rangle_{S}$
and $\exp(m)=(\frac{|m|}{2}-1,\frac{|m|}{2}+1)$.
\end{proof}

\Cref{lem:I,lem:M} complete the proof of \Cref{thm:main}.

\section{Application to Coxeter Multiarrangements of Type $A_2$}
\label{sec:application}
In this section, we obtain another expression 
of the lower derivations of bases for derivation modules for the Coxeter multiarrangements of type $A_2$, 
which was completely solved in \cite{MR2328057}, 
for a special case. 

Let $\alpha_{1}'=\alpha_1, \alpha_{2}'=\alpha_2, \alpha_{3}'=\alpha_4$ 
and we define $\AAA'=\Set{\Ker(\alpha_{1}'),\Ker(\alpha_{2}'),\Ker(\alpha_{3}')}$. 
By a linear transformation, the arrangement $\AAA'$ can be identified with the Coxeter arrangement of 
type $A_2$.
For a multiplicity $m=(m_1,m_2,m_3)$ on $\AAA'$, 
we define $D'(m)$ to be $D(\AAA',m)
=\Set{\theta \in \Der (S) |\forall i, \theta(\alpha_{i}')\in {\alpha_{i}'}^{m_{i}}S}$.

\begin{definition}
 \label{def:J'}
Let $m=(a,a,b)$ satisfy
$b\in 2\ZZ+1$, $|m|=2a+b\in 4\ZZ+3$.
Let $\mu=(1,b,a,a)$
 and $d=\frac{|\mu|}{2}-1=\frac{2a+b-1}{2}$.
We define $f'_m,g'_m,\theta'_m$ by
\begin{align*}
f'_{m}(x,y)&=f_{\mu}(x-y,x+y)-g_{\mu}(x-y,x+y)\\
&=
\sum_{i=0}^{\frac{d-1}{2}} (-1)^i 
\frac{
\dps{1+b-d}
{\frac{d-1}{2}-i}
}
{(d-2i)!(2i)!!}
(x-y)^{d-2i}(x+y)^{2i}\\
&-
\sum_{i=0}^{\frac{d-b}{2}} (-1)^{i+a} 
\frac{
\dps{1+b-d}
{\frac{d-b}{2}-i}
}
{(d-b-2i)!!(d-2i)!!(2i)!!}
(x-y)^{2i}(x+y)^{d-2i},\\
g'_{m}(x,y)&=f_{\mu}(x-y,x+y)+g_{\mu}(x-y,x+y)\\
&=
\sum_{i=0}^{\frac{d-1}{2}} (-1)^i 
\frac{
\dps{1+b-d}
{\frac{d-1}{2}-i}
}
{(d-2i)!(2i)!!}
(x-y)^{d-2i}(x+y)^{2i}\\
&+
\sum_{i=0}^{\frac{d-b}{2}} (-1)^{i+a} 
\frac{
\dps{1+b-d}
{\frac{d-b}{2}-i}
}
{(d-b-2i)!!(d-2i)!!(2i)!!}
(x-y)^{2i}(x+y)^{d-2i},\\
\theta'_{m}
&=f'_{m}(x,y)\der_{x}-g'_{m}(x,y)\der_{y}.
\end{align*}
\end{definition}

By \Cref{thm:main} and considering a coordinate transformation, 
we have $\theta'_m$ above is in $D'(m)$ if $m$ is balanced.

\begin{cor}
 \label{cor:A_2}
Let $m=(a,a,b)$ satisfy
$b\in 2\ZZ+1$, $|m|=2a+b\in 4\ZZ+3$.
If $m$ is balanced, then
$\theta'_{m}\in D'(m)$.
\end{cor}

\begin{proof}
Let $\mu = (1,b,a,a)$. 
Since $\mu$ satisfies the condition of \Cref{thm:main},
$\theta_{\mu} \in D(\mu)$, that means 
$x \mid f_{\mu}$, $y^b \mid (-g_{\mu})$, 
$(x-y)^a \mid (f_{\mu}+g_{\mu})$, and $(x+y)^a \mid (f_{\mu}-g_{\mu})$. 
For $\alpha'_1=x$,
since
$\theta'_{m}(x)=
(f_{\mu}-g_{\mu})(x-y,x+y)$, it follows that 
$((x-y)+(x+y))^a=(2^a \cdot x^a) \mid \theta'_{m}(x)$.
For $\alpha'_2=y$,
since
$\theta'_{m}(y)=
(f_{\mu}+g_{\mu})(x-y,x+y)$, it follows that 
$((x-y)-(x+y))^a=((-2)^a \cdot y^a) \mid \theta'_{m}(y)$.
For $\alpha'_3=x+y$,
since
$\theta'_{m}(x+y)=
-2g_{\mu}(x-y,x+y)$, it follows that 
$2(x+y)^b \mid \theta'_{m}(x+y)$. 
Therefore, since 
${\alpha'_1}^a \mid \theta'_{m}(\alpha'_1)$, 
${\alpha'_2}^a \mid \theta'_{m}(\alpha'_2)$ and 
${\alpha'_3}^b \mid \theta'_{m}(\alpha'_3)$, we have
$\theta'_{m} \in D'(m)$.
\end{proof}

We give some examples of $\theta'_m$.
\begin{example}
For some $m$, the derivations $\theta'_m$ are the following:
\begin{align*}
    \theta'_{(2,2,3)}&=
    \textstyle\left(-\frac{2}{3} \, x^{3} - 2 \, x^{2} y\right)\partial_x+\left(-2 \, x y^{2} - \frac{2}{3} \, y^{3}\right)\partial_y
,
\\
    \theta'_{(3,3,5)}&=
    \textstyle\left(\frac{2}{15} \, x^{5} + \frac{2}{3} \, x^{4} y + \frac{4}{3} \, x^{3} y^{2}\right)\partial_x+\left(\frac{4}{3} \, x^{2} y^{3} + \frac{2}{3} \, x y^{4} + \frac{2}{15} \, y^{5}\right)\partial_y
,
\\
    \theta'_{(4,4,3)}&=
    \textstyle\left(\frac{2}{5} \, x^{5} + \frac{2}{3} \, x^{4} y\right)\partial_x+\left(\frac{2}{3} \, x y^{4} + \frac{2}{5} \, y^{5}\right)\partial_y
,
\\
    \theta'_{(4,4,7)}&=
    \textstyle\left(-\frac{2}{105} \, x^{7} - \frac{2}{15} \, x^{6} y - \frac{2}{5} \, x^{5} y^{2} - \frac{2}{3} \, x^{4} y^{3}\right)\partial_x\\
    &\textstyle+\left(-\frac{2}{3} \, x^{3} y^{4} - \frac{2}{5} \, x^{2} y^{5} - \frac{2}{15} \, x y^{6} - \frac{2}{105} \, y^{7}\right)\partial_y
,
\\
    \theta'_{(5,5,5)}&=
    \textstyle\left(-\frac{8}{105} \, x^{7} - \frac{4}{15} \, x^{6} y - \frac{4}{15} \, x^{5} y^{2}\right)\partial_x+\left(-\frac{4}{15} \, x^{2} y^{5} - \frac{4}{15} \, x y^{6} - \frac{8}{105} \, y^{7}\right)\partial_y
,
\\
    \theta'_{(5,5,9)}&=
    \textstyle\left(\frac{2}{945} \, x^{9} + \frac{2}{105} \, x^{8} y + \frac{8}{105} \, x^{7} y^{2} + \frac{8}{45} \, x^{6} y^{3} + \frac{4}{15} \, x^{5} y^{4}\right)\partial_x\\
    &\textstyle+\left(\frac{4}{15} \, x^{4} y^{5} + \frac{8}{45} \, x^{3} y^{6} + \frac{8}{105} \, x^{2} y^{7} + \frac{2}{105} \, x y^{8} + \frac{2}{945} \, y^{9}\right)\partial_y
,
\\
    \theta'_{(6,6,3)}&=
    \textstyle\left(-\frac{4}{21} \, x^{7} - \frac{4}{15} \, x^{6} y\right)\partial_x+\left(-\frac{4}{15} \, x y^{6} - \frac{4}{21} \, y^{7}\right)\partial_y
,
\\
    \theta'_{(6,6,7)}&=
    \textstyle\left(\frac{2}{189} \, x^{9} + \frac{2}{35} \, x^{8} y + \frac{4}{35} \, x^{7} y^{2} + \frac{4}{45} \, x^{6} y^{3}\right)\partial_x\\
    &\textstyle+\left(\frac{4}{45} \, x^{3} y^{6} + \frac{4}{35} \, x^{2} y^{7} + \frac{2}{35} \, x y^{8} + \frac{2}{189} \, y^{9}\right)\partial_y
,
\\
    \theta'_{(6,6,11)}&=
    \textstyle\left(-\frac{2}{10395} \, x^{11} - \frac{2}{945} \, x^{10} y - \frac{2}{189} \, x^{9} y^{2} - \frac{2}{63} \, x^{8} y^{3} - \frac{4}{63} \, x^{7} y^{4} - \frac{4}{45} \, x^{6} y^{5}\right)\partial_x\\
    &\textstyle+\left(-\frac{4}{45} \, x^{5} y^{6} - \frac{4}{63} \, x^{4} y^{7} - \frac{2}{63} \, x^{3} y^{8} - \frac{2}{189} \, x^{2} y^{9} - \frac{2}{945} \, x y^{10} - \frac{2}{10395} \, y^{11}\right)\partial_y
,
\\
    \theta'_{(7,7,5)}&=
    \textstyle\left(\frac{2}{63} \, x^{9} + \frac{2}{21} \, x^{8} y + \frac{8}{105} \, x^{7} y^{2}\right)\partial_x+\left(\frac{8}{105} \, x^{2} y^{7} + \frac{2}{21} \, x y^{8} + \frac{2}{63} \, y^{9}\right)\partial_y

.
\end{align*}
\end{example}

\begin{remark}
For an arbitrary balanced multiplicity $m$ on $\AAA'$, 
in \cite{MR2328057}, Wakamiko has obtained a basis for $D'(m)$ explicitly 
and showed that the exponents are given by $(\lfloor \frac{|m|}{2}\rfloor,\lceil \frac{|m|}{2}\rceil)$.
Note that the coefficients of derivations of the basis in \cite{MR2328057} 
are polynomials with integer coefficients in the variables $x,y$.
If the difference of exponents is not zero, then the lower derivation of a basis is unique up to scalar.
The derivation $\theta'$ in \Cref{cor:A_2} is one of such derivations 
different from that of \cite{MR2328057}. 
The lower basis elements $\theta_\Sigma(m,\frac{|m|-1}{2})$ for $D(\AAA',m)$ in \cite{MR2328057} 
are the following: 
\begin{align*}
    \theta_{\Sigma}((2,2,3),3)&=
    \textstyle\left(x^{3} + 3 \, x^{2} y\right)\partial_x+\left(3 \, x y^{2} + y^{3}\right)\partial_y\\
    &=\textstyle-\frac{3}{2}\theta'_{(2,2,3)}
,
\\
    \theta_{\Sigma}((3,3,5),5)&=
    \textstyle\left(x^{5} + 5 \, x^{4} y + 10 \, x^{3} y^{2}\right)\partial_x+\left(10 \, x^{2} y^{3} + 5 \, x y^{4} + y^{5}\right)\partial_y\\
    &=\textstyle\frac{15}{2}\theta'_{(3,3,5)}
,
\\
    \theta_{\Sigma}((4,4,3),5)&=
    \textstyle\left(6 \, x^{5} + 10 \, x^{4} y\right)\partial_x+\left(10 \, x y^{4} + 6 \, y^{5}\right)\partial_y\\
    &=\textstyle15\theta'_{(4,4,3)}
,
\\
    \theta_{\Sigma}((4,4,7),7)&=
    \textstyle\left(x^{7} + 7 \, x^{6} y + 21 \, x^{5} y^{2} + 35 \, x^{4} y^{3}\right)\partial_x\\
    &\textstyle+\left(35 \, x^{3} y^{4} + 21 \, x^{2} y^{5} + 7 \, x y^{6} + y^{7}\right)\partial_y\\
    &=\textstyle-\frac{105}{2}\theta'_{(4,4,7)}
,
\\
    \theta_{\Sigma}((5,5,5),7)&=
    \textstyle\left(50 \, x^{7} + 175 \, x^{6} y + 175 \, x^{5} y^{2}\right)\partial_x\\
    &\textstyle+\left(175 \, x^{2} y^{5} + 175 \, x y^{6} + 50 \, y^{7}\right)\partial_y\\
    &=\textstyle-\frac{2625}{4}\theta'_{(5,5,5)}
,
\\
    \theta_{\Sigma}((5,5,9),9)&=
    \textstyle\left(x^{9} + 9 \, x^{8} y + 36 \, x^{7} y^{2} + 84 \, x^{6} y^{3} + 126 \, x^{5} y^{4}\right)\partial_x\\
    &\textstyle+\left(126 \, x^{4} y^{5} + 84 \, x^{3} y^{6} + 36 \, x^{2} y^{7} + 9 \, x y^{8} + y^{9}\right)\partial_y\\
    &=\textstyle\frac{945}{2}\theta'_{(5,5,9)}
,
\\
    \theta_{\Sigma}((6,6,3),7)&=
    \textstyle\left(15 \, x^{7} + 21 \, x^{6} y\right)\partial_x+\left(21 \, x y^{6} + 15 \, y^{7}\right)\partial_y\\
    &=\textstyle-\frac{315}{4}\theta'_{(6,6,3)}
,
\\
    \theta_{\Sigma}((6,6,7),9)&=
    \textstyle\left(490 \, x^{9} + 2646 \, x^{8} y + 5292 \, x^{7} y^{2} + 4116 \, x^{6} y^{3}\right)\partial_x\\
    &\textstyle+\left(4116 \, x^{3} y^{6} + 5292 \, x^{2} y^{7} + 2646 \, x y^{8} + 490 \, y^{9}\right)\partial_y\\
    &=\textstyle46305\theta'_{(6,6,7)}
,
\\
    \theta_{\Sigma}((6,6,11),11)&=
    \textstyle\left(x^{11} + 11 \, x^{10} y + 55 \, x^{9} y^{2} + 165 \, x^{8} y^{3} + 330 \, x^{7} y^{4} + 462 \, x^{6} y^{5}\right)\partial_x\\
    &\textstyle+\left(462 \, x^{5} y^{6} + 330 \, x^{4} y^{7} + 165 \, x^{3} y^{8} + 55 \, x^{2} y^{9} + 11 \, x y^{10} + y^{11}\right)\partial_y\\
    &=\textstyle-\frac{10395}{2}\theta'_{(6,6,11)}
,
\\
    \theta_{\Sigma}((7,7,5),9)&=
    \textstyle\left(490 \, x^{9} + 1470 \, x^{8} y + 1176 \, x^{7} y^{2}\right)\partial_x\\
    &\textstyle+\left(1176 \, x^{2} y^{7} + 1470 \, x y^{8} + 490 \, y^{9}\right)\partial_y\\
    &=\textstyle15435\theta'_{(7,7,5)}

.
\end{align*}
\end{remark}

\bibliography{by-mr,by-arxiv}
\bibliographystyle{amsplain-url}

\end{document}